\documentclass[a4, 12pt, reqno]{amsart}
\usepackage{fullpage}
\usepackage[title]{appendix}
\usepackage{graphicx}
\usepackage{subcaption}
\usepackage{amsthm}
\usepackage{mathrsfs} %%pour le style calligraphie en mathscr
\usepackage{amsfonts}
\usepackage{amssymb}
\usepackage{diagbox}
\usepackage{mathtools}
\usepackage{longtable} %tables with pagebreak
\usepackage{MnSymbol}
\usepackage{indentfirst}
\usepackage[all]{xy}
\usepackage[colorlinks=true,linkcolor=blue]{hyperref}
\usepackage{ytableau}
\usepackage{tikz-cd}
\usepackage{enumerate}
\theoremstyle{plain}
\newtheorem{theorem}{Theorem}[section]

\newtheorem{lemma}[theorem]{Lemma}
\newtheorem{corollary}[theorem]{Corollary}
\newtheorem{proposition}[theorem]{Proposition}

\theoremstyle{remark}

\makeatletter

\newcommand{\Rmnum}[1]{\expandafter\@slowromancap\romannumeral #1@}
\makeatother

\def\ri{{\rm i}}

\def\bQ{\mathbb Q}

\def\bN{\mathbb N}
\def\bZ{\mathbb Z}
\def\bC{{\mathbb C}}

\def\cH{{\mathcal H}}

\newmuskip\pFqmuskip

\numberwithin{equation}{section}
\allowdisplaybreaks %allow page breaks in align enviroment

\begin{document}

\title[Divisor functions and the half Appell sums]{Restricted divisor functions and half Appell sums \\ in higher-level}
\author{Meng-Juan Tian and Nian Hong Zhou}

\address{M.J. Tian \& N. H. Zhou: School of Mathematics and Statistics, The Center for Applied Mathematics of Guangxi, Guangxi Normal University, Guilin 541004, Guangxi, PR China}
\email{mengjuantian@outlook.com; nianhongzhou@outlook.com}%
\thanks{Nian Hong Zhou is the corresponding author}
\thanks{This paper was partially supported by the National Natural Science Foundation of China (No. 12301423).}%

\begin{abstract}
We examine the value distributions of coefficients in certain $q$-series related to half Appell sums in higher-level and the first moment of the Garvan's $k$-rank of partitions. We prove that these coefficients equal certain restricted divisor functions and can take any nonnegative integer value infinitely many times. As applications, we confirm a conjecture of Xiong on the coefficients of a half Lerch sum and a conjecture of Garvan and Jennings-Shaffer on the nonnegativity of spt-crank-type partitions.
\end{abstract}
\maketitle

\section{Introduction}\label{sec1}
The level $1$ \emph{Appell function} or \emph{Lerch sum} was studied by Appell (1884), Lerch (1892) and others, which is defined formally as the following
\begin{equation}\label{eqm0}
A(z;\tau)=w^{1/2}\sum_{n \in \bZ}\frac{(-1)^{n}q^{\binom{n+1}{2}}}{1-wq^n},
\end{equation}
where $\Im(\tau)>0$, $z\in\bC$, $w=e^{2\pi\ri z}$ and $q=e^{2\pi\ri \tau}$. This function plays an important role in the theory of mock modular forms as shown by Zwegers in his remarkable thesis. In the special case that $z=0$, a half-sum of \eqref{eqm0} reads as
\begin{align}\label{eqm1}
\cH(q):=\sum_{n\ge 0} h(n)q^n:=\sum_{n\ge 1}\frac{(-1)^{n+1} q^{\binom{n+1}{2}}}{1-q^n}.
\end{align}
This series was considered and studied by Andrews-Chan-Kim-Osburn \cite{MR3505316} in their work on the first positive rank and crank moments for overpartitions. Moreover, they \cite[Lemma 2.2]{MR3505316} established the following identity
\begin{align}\label{eqm2}
\cH(q)=\sum_{\substack{j\ge 1\\ 0\le r<j}}q^{j(j+r)}(1+q^j).
\end{align}
From which we immediately see that $h(n)$ are always nonnegative.

\medskip

Motivated by the work of Andrews-Dyson-Hickerson \cite{MR928489} on the value distributions of the coefficients $S(n)$ of the $\sigma$ function in Ramanujan's Lost Notebook V (see \cite{MR849849}):
\begin{align*}
\sigma(q):=\sum_{n\ge 0}S(n)q^n&=\sum_{n\ge 0}\frac{q^{\binom{n+1}{2}}}{(-q;q)_n}=\sum_{\substack{n\ge 0\\ |j|\le n}}(-1)^{n+j}q^{n(3n+1)/2-j^2}(1-q^{2n+1}),
\end{align*}
where $(a;q)_{n}:=\prod_{0\le j<n}(1-aq^j)$ for any $a\in\bC$, $|q|<1$ and $n\in\bN_0\cup\{\infty\}$, 
Xiong \cite{MR3704372} and Chen \cite{MR3922597} studied the value distribution of $h(n)$ and established many asymptotic results for the counting function
$$S_i(x):=\#\{1\le n\le x: h(n)=i\},$$
that is, the number of $n \le x$ such that $h(n) = i$, by using the identity \eqref{eqm2}. For example, Chen \cite[Theorem 1.2]{MR3922597} proved that almost all the coefficients $h(n)$ vanish. In particular, he proved that
\begin{align}\label{caf}
S_0(x)=x+O\left(\frac{x}{(\log x)^{\delta}(\log\log x)^{3/2}}\right),
\end{align}
for all $x>3$, where $\delta=1-\frac{1+\log\log 2}{\log 2}=0.086071\cdots$. Xiong \cite[Conjecture 1]{MR3704372} conjectured that
\begin{equation}\label{eqcx}
\limsup_{n\to \infty} h(n)=+\infty.
\end{equation}
This conjecture analogs Andrews's conjecture (see \cite{MR1540966}): $\limsup_{n\to\infty}|S(n)|=+\infty$. Andrews's conjecture was soon proved by himself, Dyson and Hickerson \cite{MR928489} by related $\sigma(q)$ to the arithmetic of quadratic field $\bQ(\sqrt{6})$.

\medskip

Our original motivation for this paper is to prove the above conjecture of Xiong \eqref{eqcx}. We in fact investigate the value distributions of the coefficients of a more general \emph{half Appell sums in higher-level} $\cH_{k,m}(q)$. Define for any odd integer $k\ge 1$ and any integer $m\ge 0$ that
\begin{align}\label{eqm3}
\cH_{m,k}(q):=\sum_{n\ge 0}h_{m,k}(n)q^n=\sum_{n\ge 1}\frac{(-1)^{n-1} q^{k\binom{n}{2}+mn}}{1-q^n}.
\end{align}
Clearly, $h(n)=h_{1,1}(n)$. The sum \eqref{eqm3} for $\cH_{m,k}(q)$ is almost a half sums of level $k$ \emph{Appell function} when $m=k$. Recall that  the level $k$ Appell functions is defined by
$$
A_k(z;\tau)=w^{k/2}\sum_{n \in \bZ}\frac{(-1)^{kn}q^{k\binom{n+1}{2}}}{1-wq^n},
$$
see Zwegers's \cite{MR3048661} for the more details on higher level Appell functions.

\medskip

The function $\cH_{m,2k-1}(q)$ also related to the first moments of Garvan $k$-ranks for partitions.
Define for any integer $m\in\bZ$ and $k\ge 1$ that
\begin{align*}
\sum_{n\ge 0}N_k(m,n)q^n=\frac{1}{(q;q)_\infty}\sum_{n\ge 1}(-1)^{n-1}q^{n((2k-1)n-1)/2+|m|n}(1-q^n).
\end{align*}
It is well-known that $N_1(m,n)=M(m,n)$ is the Andrews-Garvan-Dyson's
crank function and $N_2(m,n)=N(m,n)$ is Dyson's rank function. Garvan \cite{MR1291125} proved that for any integer $k\ge 2$, $N_k(m, n)$ is the number of partitions of $n$ into at least $(k-1)$ successive Durfee squares with \emph{$k$-rank} equal to $m$. For the details of the combinatorial interpretation of $N_k(m,n)$, see \cite[Theorem (1.12)]{MR1291125}. Define that
$$N_{k}^\dag(m, n)=\sum_{\ell \ge m+(1-k)}\ell N_k(\ell,n),$$
the first partial moments of Garvan $k$-rank for partitions. Clearly, when $k=1$ and $k=2$,
$$N_{k}^\dag\left(k, n\right)=\sum_{\ell \ge 1}\ell N_k(\ell,n)$$
are the first positive crank moment and rank moment for partitions, respectively.  See Andrews-Chan-Kim \cite{MR2971698} for details.  Moreover, for any $m\ge k-1$, by elementary arguments we obtain
\begin{align*}
\sum_{n\ge 0}N_{k}^\dag(m, n)q^n-\frac{m-k}{(q;q)_\infty}\sum_{n\ge 1}(-1)^{n-1}q^{(2k-1)\binom{n}{2}+mn}=\frac{1}{(q;q)_\infty}\sum_{n\ge 1}\frac{(-1)^{n-1}q^{(2k-1)\binom{n}{2}+mn}}{1-q^n}.
\end{align*}

\medskip

The first result of this paper is the following restricted divisor sum formula.
\begin{theorem}\label{main}For any integer $m\ge 0$ and any odd integer $k\ge 1$ we have
\begin{align*}
\cH_{m,k}(q)=\sum_{\ell\ge 1}\sum_{\substack{j\in\bZ\\ m+k(\ell-1)/2\le j<2m+k(2\ell-1)}}q^{j\ell}.
\end{align*}
Moreover,
\begin{equation*}
h_{m,k}(n)=\sum_{\substack{d\mid n\\  u_{m,k}(n)  < 2kd\le 2u_{m,k}(n)}}1,
\end{equation*}
where $u_{m,k}(n)=\sqrt{2kn+(m-k/2)^2}-(m-k/2)$. In particular, $h_{m,k}(n)\ge 0$ for all $n$.
\end{theorem}
Let $m\ge 0$ be any integer and $k\ge 1$ be any odd integer. Define for any $x>1$ and any integer $\ell\ge 0$ that
$$S_{m,k}(\ell; x):=\#\{0\le n\le x: h_{m,k}(n)=\ell\}.$$
Based on Theorem \ref{main}, we prove the following theorem.
\begin{theorem}\label{main0}For all $x\ge 2$, any integers $m,\ell\ge 0$ and any odd integer $k\ge 1$, we have
\begin{align*}
S_{m,k}(\ell;x)\gg \begin{cases}
(\log x)^{-1}x^{\frac{1}{\ell+1}}\;\;\text{for}\;\; \ell\in\{0,1\},\\
 (\log x)^{-2}x^{\frac{1}{\ell-1}}\;\;\text{for}\;\; \ell\ge 2,
 \end{cases}
\end{align*}
where the implied constant depends only on $m$, $k$ and $\ell$. In other words, $h_{m,k}(n)$ can take any nonnegative integer infinity many times.
\end{theorem}
Consequently, we obtain the following corollary.
\begin{corollary}
For any integer $m\ge 0$ and any odd integer $k\ge 1$, we have
$$\limsup_{n\to \infty}h_{m,k}(n)=+\infty,$$
In particular, the conjecture \eqref{eqcx} of Xiong is true.
\end{corollary}
Based on Theorem \ref{main} and a work of Ford \cite[Corollary 2]{MR2434882} on the distribution of integers with a divisor in a given interval, we can further obtain the following asymptotic formula for $S_{m,k}(0; x)$, which generalizes Chen's asymptotic formula \eqref{caf}.
\begin{theorem}\label{corm}
Let $m\ge 0$ be any integer and $k\ge 1$ be any odd integer. For all $x\ge 3$, we have
$$S_{m,k}(0; x)=x+O\left((\log x)^{-\delta} (\log\log x)^{-3/2}x\right),$$
where the implied constant depends only on $m$ and $k$.
\end{theorem}

\medskip

Finally, applications of Theorem \ref{main} include the following inequalities for $h_{m,k}(n)$.
\begin{proposition}\label{pro1}For all integers $m, n \geq 0$, and all odd integers $k \geq 1$, we have
$$h_{m,k}(2n)\ge h_{m,k}(n).$$
\end{proposition}
Based on Proposition \ref{pro1}, we can easily prove a conjecture on the positivity of certain spt-crank-type functions stated in the conclusion section of Garvan and Jennings-Shaffer \cite{MR3519478}. We note that their work was motivated by Andrews--Garvan--Liang's study \cite{MR2994105} of the two-variable spt-crank-type series for partitions:
$$S(z,q):=\sum_{n\ge 0}\sum_{m\in\mathbb{Z}}N_S(m,n)z^mq^n=\sum_{n\ge 1}\frac{(q^{n+1};q)_\infty  q^n}{(zq^n;q)_\infty(z^{-1}q^n;q)_\infty}.$$
In \cite[Theorem 5.1]{MR3519478}, they showed that $N_S(m,n)\ge 0$ for all integers $m$ and $n$.
Garvan and Jennings-Shaffer \cite[Equations (2.1)--(2.8)]{MR3519478} introduced eight two-variable spt-crank-type series $S_X(z,q)$, along with their corresponding spt-crank-type functions $M_X(m,n)$, defined by
$$
S_X(z,q) := \sum_{n \geq 0} \sum_{m \in \mathbb{Z}} M_X(m,n) z^m q^n.
$$
In the conclusion section, they point out that it is possible to interpret $M_X(m, n)$ as a statistic defined on the smallest parts that each spt-type function ${\rm spt}_X(n)$ counts. They note that six of these spt-crank-type functions $M_X(m,n)$ are nonnegative, except for the two spt-crank-type functions $M_{C_1}(m,n)$ and $M_{C_5}(m,n)$, which are defined by
$$S_{C_1}(z,q):=\sum_{n\ge 0}\sum_{m\in \mathbb{Z}} M_{C_1}(m,n) z^m q^n = \sum_{n\ge 1} \frac{q^n(q^{2n+1};q^2)_{\infty} (q^{n+1};q)_{\infty}}{(zq^n;q)_{\infty}(z^{-1}q^n;q)_{\infty}} $$
and
$$S_{C_5}(z,q):=\sum_{n\ge 0}\sum_{m\in \mathbb{Z}} M_{C_5}(m,n) z^m q^n = \sum_{n\ge 1} \frac{q^{\frac{n(n+1)}{2}}(q^{2n+1};q^2)_{\infty} (q^{n+1};q)_{\infty}}{(zq^n;q)_{\infty}(z^{-1}q^n;q)_{\infty}}. $$
They emphasized that numerical evidence suggests that both $M_{C_1}(m,n)$ and $M_{C_5}(m,n)$ are also nonnegative, and posed the problem of finding nice combinatorial interpretations for $M_{C_1}(m,n)$ and $M_{C_5}(m,n)$ that would prove their nonnegativity. In \cite{MR3745072}, Jang and Kim employed the circle method to prove that for any fixed integer $m$, both $M_{C_1}(m,n)$ and $M_{C_5}(m,n)$ are positive for all sufficiently large $n$. Recently, He and Liu \cite{hl2025} used the lattice point counting method to prove that both $M_{C_1}(m,n)$ and $M_{C_5}(m,n)$ are nonnegative for all integers $m$ and $n$.
\medskip

In what follows, we provide our proof of the nonnegativity of $M_{C_1}(m,n)$ and $M_{C_5}(m,n)$. By the first formula in \cite[Corollary 2.10]{MR3519478}, we have
$$M_{C_1}(m,n) = M_{C_5}(m,n) + N_{S}(m,n/2),$$
where we define $N_{S}(m,x):= 0$ for any $x \notin \mathbb{Z}$. Thus, it is clear that if $M_{C_5}(m,n) \geq 0$, then $M_{C_1}(m,n) \geq 0$, because $N_{S}(m,n/2) \geq 0$ for all $m,n$. Next, by \cite[Equation (2.2)]{MR3745072}, we have
\begin{align*}
\sum_{n\geq 0}M_{C_5}(m,n)q^n&=\frac{1}{(q^2;q^2)_{\infty}}\sum_{n\geq 1}(-1)^{n-1}\left(\frac{q^{\frac{n(n+1)}{2}+|m|n}}{1-q^n}-\frac{q^{n^2+n+2|m|n}}{1-q^{2n}}\right)\\
&=\frac{1}{(q^2;q^2)_{\infty}}\left(\sum_{n\ge 0}h_{|m|+1,1}(n)q^n-\sum_{n\ge 0}h_{|m|+1,1}(n)q^{2n}\right)\\
&=\frac{1}{(q^2;q^2)_{\infty}}\sum_{n\ge 0}h_{|m|+1,1}(2n+1)q^{2n+1}+\sum_{n\ge 0}\left(h_{|m|+1,1}(2n)-h_{|m|+1,1}(n)\right)q^{2n}.
\end{align*}
Therefore, by Proposition \ref{pro1}, we obtain the nonnegativity of $M_{C_5}(m,n)$. This completes the proof.

\section{The proofs}\label{sec2}
In this section we prove all the results stated in the Section \ref{sec1}. We prove Theorem \ref{main} and Proposition \ref{pro1} in Subsection \ref{subsec21} by using elementary series manipulations. In the remaining Subsection \ref{subsec22} we prove Theorem \ref{main0} and Theorem \ref{corm}, which will rely on the prime number theorem and the work \cite[Corollary 2]{MR2434882} of Ford on the distribution of integers with a divisor in a given interval.
\subsection{Proofs of Theorem \ref{main} and Proposition \ref{pro1}}\label{subsec21}
\begin{proof}[Proof of Theorem \ref{main}]We compute that
\begin{align}\label{eqh1}
\cH_{m,k}(q)=&\sum_{\ell\ge 1}(-1)^{\ell-1} \sum_{j\ge m}q^{\ell(j+k(\ell-1)/2)}\nonumber\\
=&\sum_{\substack{\ell\ge 1\\ \ell\equiv 1\pmod 2}}\sum_{j\ge m+k(\ell-1)/2}q^{j\ell}-\sum_{\ell\ge 1}\sum_{j\ge m}q^{\ell(2j+(2\ell-1)k)}\nonumber\\
=&\Big(\sum_{\ell\ge 1}-\sum_{\substack{\ell\ge 1\\ \ell\equiv 0\pmod 2}}\Big)\sum_{j\ge m+k(\ell-1)/2}q^{j\ell}-\sum_{\ell\ge 1}\sum_{\substack{j\ge 2m+k(2\ell-1)\\ j\equiv k\pmod 2}}q^{j\ell}.
\end{align}
Since $k\equiv 1\pmod 2$, we have
\begin{align}\label{eqh2}
\sum_{\ell\ge 1}\sum_{\substack{j\ge 2m+k(2\ell-1)\\ j\equiv k\pmod 2}}q^{j\ell}&=\sum_{\ell\ge 1}\sum_{j\ge 2m+k(2\ell-1)}q^{j\ell}-\sum_{\ell\ge 1}\sum_{\substack{j\ge 2m+k(2\ell-1)\\ j\equiv 0\pmod 2}}q^{j\ell}\nonumber\\
&=\sum_{\ell\ge 1}\sum_{j\ge 2m+k(2\ell-1)}q^{j\ell}-\sum_{\ell\ge 1}\sum_{j\ge m+k(2\ell-1)/2}q^{2j\ell}\nonumber\\
&=\sum_{\ell\ge 1}\sum_{j\ge 2m+k(2\ell-1)}q^{j\ell}-\sum_{\substack{\ell\ge 1\\ \ell\equiv 0\pmod 2}}\sum_{j\ge m+k(\ell-1)/2}q^{j\ell}.
\end{align}
Therefore, by substituting \eqref{eqh2} into \eqref{eqh1}, we obtain
\begin{align*}
\cH_{m,k}(q)&=\sum_{\ell\ge 1}\sum_{\substack{j\in\bZ\\ m+k(\ell-1)/2\le j<2m+k(2\ell-1)}}q^{j\ell}.
\end{align*}
Thus for any $n\in\bN$, we have
\begin{align*}
h_{m,k}(n)&=\#\{(j,\ell)\in\bN^2: j\ell=n,\; (m-k/2)+k\ell/2\le j< 2(m-k/2)+2k\ell\}\\
&=\#\{\ell\mid n:  (m-k/2)+k\ell/2\le j=n/\ell< 2(m-k/2)+2k\ell\}\\
&=\#\left\{\ell\mid n:  u_{m,k}(n)< 2k\ell\le 2u_{m,k}(n)\right\},
\end{align*}
where $u_{m,k}(n)=\sqrt{2kn+(m-k/2)^2}-(m-k/2)$. This completes the proof.
\end{proof}
To prove Proposition \ref{pro1}, we in fact prove the following more precise result on the difference between $h_{m,k}(2n)$ and $h_{m,k}(n)$, which is stated as Proposition \ref{prom}.
\begin{proposition}\label{prom}Let $m\ge 0$ be any integer and $k\ge 1$ any odd integer. Let $n\ge 1$ be any odd integer and $r\in\bN_0$. Then, we have
\begin{align*}
h_{m,k}(2^{r+1}n)-h_{m,k}(2^{r}n)
=&\#\left\{\ell\mid n: 2^{-1-r}u_{m,k}(2^{r+1}n) < 2k\ell\le 2^{-r}u_{m,k}(2^{r}n)\right\}\\
&+\#\left\{\ell\mid n: 2u_{m,k}(2^{r}n) < 2k\ell\le 2u_{m,k}(2^{r+1} n)\right\}.
\end{align*}
In particular, for all $n\in\bN$ we have
$$h_{m,k}(2n)\ge h_{m,k}(n).$$
\end{proposition}

\begin{proof}
Let $m\ge 0$ be any integer and $k\ge 1$ be any odd integer. By Theorem \ref{main} and note that $n$ is an odd integer, we have
\begin{align*}
h_{m,k}(2^{r}n)=&\#\left\{\ell\mid 2^r n: u_{m,k}(2^rn)  < 2k\ell\le 2u_{m,k}(2^rn)\right\}\\
=&\sum_{0\le s\le r}\#\left\{\ell\mid n: u_{m,k}(2^{r}n) < 2k\cdot 2^s \ell\le 2u_{m,k}(2^{r} n)\right\}\\
=&\sum_{0\le s\le r}\#\left\{\ell\mid n: 2^{-s}u_{m,k}(2^{r}n) < 2k\ell\le 2^{1-s}u_{m,k}(2^{r} n)\right\}\\
=&\#\left\{\ell\mid n: 2^{-r}u_{m,k}(2^{r}n) < 2k\ell\le 2u_{m,k}(2^{r} n)\right\}.
\end{align*}
This implies
\begin{align*}
h_{m,k}(2^{r+1}n)-h_{m,k}(2^{r}n)=&\#\left\{\ell\mid n: 2^{-1-r}u_{m,k}(2^{r+1}n) < 2k\ell\le 2u_{m,k}(2^{r+1} n)\right\}\\
&-\#\left\{\ell\mid n: 2^{-r}u_{m,k}(2^{r}n) < 2k\ell\le 2u_{m,k}(2^{r} n)\right\}\\
=&\#\left\{\ell\mid n: 2^{-1-r}u_{m,k}(2^{r+1}n) < 2k\ell\le 2^{-r}u_{m,k}(2^{r}n)\right\}\\
&+\#\left\{\ell\mid n: 2u_{m,k}(2^{r}n) < 2k\ell\le 2u_{m,k}(2^{r+1} n)\right\},
\end{align*}
by note that
$$2u_{m,k}(2^{r+1} n)\ge 2u_{m,k}(2^{r} n)\ge 2^{-r}u_{m,k}(2^{r}n)\ge 2^{-1-r}u_{m,k}(2^{r+1}n),$$
for all $n, k\ge 1$ and $m, r\ge 0$. This completes the proof.
\end{proof}

\subsection{Proofs of Theorem \ref{main0} and Theorem \ref{corm}}\label{subsec22}
We begin with a normalization of $u_{m,k}(n)$ for the simplification of our discussions.
We write
$$u_{m,k}(n)=\sqrt{2kn}w_{m,k}(k^{-1}n),$$
where
$$w_{m,k}(n)=\sqrt{1+(8n)^{-1}(1-2m/k)^2}+(8n)^{-1/2}(1-2m/k).$$
Clearly, $\lim_{n\to+\infty}w_{m,k}(n)=1$. In particular, for $n\ge 100(1-2m/k)^2$, we have
\begin{align}\label{eq30}
|w_{m,k}(n)-1|<0.04.
\end{align}
By Theorem \ref{main} we have
\begin{align}\label{eqmmm}
h_{m,k}\left(n\right)&=\#\left\{\ell \mid n:  \sqrt{2kn}w_{m,k}(k^{-1}n)< 2k\ell\le 2\sqrt{2kn}w_{m,k}(k^{-1}n)\right\}\nonumber\\
&=\#\left\{\ell\mid n:  2^{-1/2}k^{-1/2}w_{m,k}(k^{-1}n)< n^{-1/2}\ell\le 2^{1/2}k^{-1/2}w_{m,k}(k^{-1}n)\right\}.
\end{align}
For all positive integers $n$ such that $\gcd(k,n)=1$, we further obtain
\begin{align}\label{eq31}
h_{m,k}\left(kn\right)&=\#\left\{\ell\mid kn:  2^{-1/2}w_{m,k}(n)< n^{-1/2}\ell\le 2^{1/2}w_{m,k}(n)\right\}\nonumber\\
&=\sum_{d\mid k}\#\left\{\ell\mid n:  2^{-1/2}d^{-1}w_{m,k}(n)< n^{-1/2}\ell\le 2^{1/2}d^{-1}w_{m,k}(n)\right\}.
\end{align}
To prove Theorem \ref{main0}, we first prove the following Lemmas \ref{lem22}-\ref{lem24}, which construct infinite subsets of $\{n \in \mathbb{N} : h_{m,k}(n) = \ell\}$ for each integer $\ell \ge 0$.
\begin{lemma}\label{lem22}Let $k$ be a positive odd integer and $m$ be any nonnegative integer. Then, for all prime $p\ge \max(3k^2, 10|1-2m/k|)$ we have $h_{m,k}(kp)=0$.
\end{lemma}
\begin{proof}
Since prime $p\ge \max(3k^2, 10|1-2m/k|)$, we have $\gcd(k,p)=1$ and
\begin{align*}
h_{m,k}\left(k p \right)=\sum_{d\mid k}\#\left\{i\in\{0,1\}:  2^{-1/2}d^{-1}w_{m,k}(p)< p^{i-1/2}\le 2^{1/2}d^{-1}w_{m,k}(p)\right\},
\end{align*}
by using \eqref{eq31}. Moreover, by using \eqref{eq30},
$$2^{-1/2}d^{-1}w_{m,k}(p)\ge 2^{-1/2}k^{-1}(1-0.04)> p^{-1/2}\;\;\text{and}\;\; 2^{1/2}d^{-1}w_{m,k}(p)<2^{1/2}(1+0.04)<p^{1/2},$$
holds for all $d\mid k$. Thus for any $d\mid k$, there does not exist $i\in\{0,1\}$ such that the inequalities
$$2^{-1/2}d^{-1}w_{m,k}(p)< p^{i-1/2}\le 2^{1/2}d^{-1}w_{m,k}(p)$$
hold, that is, $h_{m,k}\left(kp\right)=0$, which completes the proof.
\end{proof}

\begin{lemma}\label{lem23} Let $k$ be a positive odd integer and $m$ be any nonnegative integer. Then, for all prime $p\ge \max(2k, 10|1-2m/k|)$ we have $h_{m,k}(kp^2)=1$.
\end{lemma}
\begin{proof}
Since prime $p\ge \max(2k, 10|1-2m/k|)$, we have $\gcd(k,p)=1$ and
\begin{align*}
h_{m,k}\left(k p^2 \right)=\sum_{d\mid k}\#\left\{i\in\{0,1,2\}:  2^{-1/2}d^{-1}w_{m,k}(p^2)< p^{i-1}\le 2^{1/2}d_0^{-1}w_{m,k}(p^2)\right\},
\end{align*}
by using \eqref{eq31}. Moreover, by using \eqref{eq30},
$$2^{-1/2}d^{-1}w_{m,k}(p^2)\ge 2^{-1/2}k^{-1}(1-0.04)> p^{-1}\;\;\text{and}\;\; 2^{1/2}d^{-1}w_{m,k}(p^2)<\frac{2}{3}(1-0.04)<p^0,$$
holds for all $d\ge 3$ with $d\mid k$. Thus for all prime $p\ge \max(2k, 10|1-2m/k|)$,
\begin{align*}
h_{m,k}\left(kp^2\right)&=\sum_{d\in \{1\}}\#\left\{i\in\{0,1,2\}:  2^{-1/2}d^{-1}w_{m,k}(p^2)< p^{i-1}\le 2^{1/2}d^{-1}w_{m,k}(p^2)\right\}\\
&=\#\left\{i\in\{1\}:  2^{-1/2}w_{m,k}(p^2)< p^{i-1}\le 2^{1/2}w_{m,k}(p^2)\right\}=1,
\end{align*}
which completes the proof.
\end{proof}

\begin{lemma}\label{lem24} Let $t\ge 1, m\ge 0$ be any integers and $k\ge 1$ be an odd integer. Then, for all primes $q>p\ge \max(2k, 10|1-2m/k|)$ with $(q/p)^{t/2}\le 0.96\sqrt{2}\approx1.357645$, we have $h_{m,k}(kp^tq^t)=1+t$.
\end{lemma}
\begin{proof}
Since $q>p\ge \max(2k, 10|1-2m/k|)$ are primes and $q/p=:\lambda$, we have $\gcd(k, p^tq^t)=1$. Thus, the use of \eqref{eq31} yields
\begin{align*}
h_{m,k}\left(kp^tq^t\right)
&=\sum_{d\mid k}\#\left\{\ell\mid p^tq^t: 2^{-1/2}d^{-1}w_{m,k}(p^tq^t)< p^{-t/2}q^{-t/2}\ell\le 2^{1/2}d^{-1}w_{m,k}(p^tq^t)\right\}\\
&=\sum_{d\mid k}\#\left\{(i,j)\in\{0,1,\ldots,t\}^2: \frac{2^{-1/2}}{d}w_{m,k}(p^tq^t)< p^{i-t/2}q^{j-t/2}\le \frac{2^{1/2}}{d}w_{m,k}(p^tq^t)\right\}\\
&=\sum_{d\mid k}\#\left\{(i,j)\in\{0,1,\ldots,t\}^2: \frac{2^{-1/2}}{\lambda^{j-t/2}d}w_{m,k}(p^tq^t)< p^{i+j-t}\le \frac{2^{1/2}}{\lambda^{j-t/2}d}w_{m,k}(p^tq^t)\right\}.
\end{align*}
Moreover, by using \eqref{eq30},
$$\frac{2^{-1/2}}{\lambda^{j-t/2}d}w_{m,k}(p^tq^t)>\frac{2^{-1/2}}{\lambda^{t/2}k}(1-0.04)\ge\frac{2^{-1/2}}{0.96\sqrt{2}k}\cdot 0.96 =(2k)^{-1}\ge p^{-1},$$
and
$$\frac{2^{1/2}}{\lambda^{j-t/2}d}w_{m,k}(p^tq^t)<\frac{1}{3}\cdot 2^{1/2}(1+0.04)\lambda^{t/2}\le \frac{1}{3}\cdot 2^{1/2}(1+0.04)\cdot 2^{1/2}(1-0.04)<1=p^0,$$
holds for all $0\le j\le t$ and $d\ge 3$ with $d\mid k$. Thus
\begin{align*}
h_{m,k}\left(kp^tq^t\right)&=\#\left\{(i,j)\in\{0,1,\ldots,t\}^2: \frac{2^{-1/2}}{\lambda^{j-t/2}}w_{m,k}(p^tq^t)< p^{i+j-t}\le \frac{2^{1/2}}{\lambda^{j-t/2}}w_{m,k}(p^tq^t)\right\}\\
&=\#\left\{(i,j)\in\{0,1,\ldots,t\}^2: i+j-t=0\right\}=1+t,
\end{align*}
which completes the proof.
\end{proof}

\begin{proof}[Proof of Theorem \ref{main0}]Recall that the prime number theorem state that
$$\sum_{p\le x,\; p~is~prime }1=\frac{x}{\log x}\left(1+O\left((\log x)^{-1}\right)\right),$$
as $x\to\infty$. This immediately yields  
$$\sum_{\alpha x<p\le x,\; p~is~prime }1=(1-\alpha)\frac{x}{\log x}\left(1+O\left((\log x)^{-1}\right)\right),$$
for any given $\alpha\in (0,1)$, as $x\to\infty$.
Using Lemmas \ref{lem22} and \ref{lem23}, we have
$$S_{m,k}(0; x)=\sum_{n\le x}{\bf 1}_{h_{m,k}(n)=0}\gg \sum_{x/(2k)< p\le x/k}{\bf 1}_{h_{m,k}(kp)=0}\gg \sum_{x/(2k)<p\le x/k,\; p~is~prime}1\gg \frac{x}{\log x},$$
and
$$S_{m,k}(1; x)=\sum_{n\le x}{\bf 1}_{h_{m,k}(n)=1}\gg \sum_{x/(2k)< p^2\le x/k}{\bf 1}_{h_{m,k}(kp^2)=1}\gg \sum_{(x/(2k))^{1/2}< p\le (x/k)^{1/2},\; p~is~prime}1\gg \frac{\sqrt{x}}{\log x},$$
as $x\to\infty$, where the implied constant depends only on $m$ and $k$. For $\ell=t+1$, using Lemma \ref{lem24} and prime number theorem, we obtain
\begin{align*}
S_{m,k}(\ell; x)=\sum_{n\le x}{\bf 1}_{h_{m,k}(n)=1+t}&\gg \sum_{\substack{p^tq^t\le x/k \\1<(q/p)^{t/2}\le 1.3\\ q>p\ge \max(2k, 10|1-2m/k|)~are~primes}}{\bf 1}_{h_{m,k}(kp^tq^t)=1+t}\\
&\gg  \sum_{\substack{p~is~prime \\ \frac{1}{2}(x/k)^{\frac{1}{2t}}<p\le (x/k)^{\frac{1}{2t}}}}\#\{q~is~prime: p<q\le 1.3^{2/t}p\},
\end{align*}
that is
\begin{align*}
S_{m,k}(\ell; x)&\gg (1.3^{2/t}-1)\sum_{\frac{1}{2}(x/k)^{\frac{1}{2t}}<p\le (x/k)^{\frac{1}{2t}},\; p~is~prime}\frac{p}{\log p}\\
&\gg  \frac{x^{\frac{1}{2t}}}{\log x}\sum_{\frac{1}{2}(x/k)^{\frac{1}{2t}}<p\le (x/k)^{\frac{1}{2t}},\; p~is~prime}1\gg \frac{x^{{1}/{t}}}{(\log x)^2},
\end{align*}
where the implied constant depends only on $\ell$, $m$ and $k$. This completes the proofs.
\end{proof}

We finally give the proof for Theorem \ref{corm}.
\begin{proof}[Proof of Theorem \ref{corm}]
For all $x\ge 200k(1-2m/k)^2$ and all $x/2< n\le x$, by \eqref{eqmmm} we have
\begin{align}\label{eqff}
h_{m,k}\left(n\right)&=\#\left\{\ell \mid n:  \frac{1}{\sqrt{2k}}n^{1/2}w_{m,k}(k^{-1}n)< \ell\le \frac{2}{\sqrt{2k}}n^{1/2}w_{m,k}(k^{-1}n)\right\}\nonumber\\
&\le \#\left\{\ell \mid n:  \frac{1}{\sqrt{2k}}(x/2)^{1/2}(1-0.04)< \ell\le \frac{2}{\sqrt{2k}}x^{1/2}(1+0.04)\right\}\nonumber\\
&\le \#\left\{\ell \mid n:  3^{-1}k^{-1/2}x^{1/2}< \ell\le 2k^{-1/2}x^{1/2}\right\}.
\end{align}
Let $H(x,y,z)$ be the number of positive integers not exceeding $x$ which have a divisor in the interval $(y,z]$. Recall that Ford \cite[Corollary 2]{MR2434882} states: If $c>1$ and $1/(c-1)\le y\le x/c$, then
$$H(x,y,cy)\asymp \frac{x}{(\log Y)^{\delta}(\log\log Y)^{3/2}}\;\;\; (Y=\min(y, x/y)+3),$$
where $\delta=1-\frac{1+\log\log 2}{\log 2}=0.086071\cdots$.  By combining this with \eqref{eqff}, we obtain
\begin{align*}
\sum_{x/2< n\le x}{\bf 1}_{h_{m,k}\left(n\right)>0}&\le \sum_{n\le x}{\bf 1}_{\#\left\{\ell \mid n:  3^{-1}k^{-1/2}x^{1/2}< \ell\le 2k^{-1/2}x^{1/2}\right\}>0}\\
&\le H\left(x, 3^{-1}k^{-1/2}x^{1/2},  2k^{-1/2}x^{1/2}\right) \asymp \frac{x}{(\log x)^{\delta}(\log\log x)^{3/2}},
\end{align*}
for all $x\ge 200k(1-2m/k)^2$, where the implied constant depends only on $k$. Thus, for all $x>0$ such that $\sqrt{x}\ge 200k(1-2m/k)^2$,
we have
\begin{align*}
\sum_{n\le x}{\bf 1}_{h_{m,k}\left(n\right)>0}&=\sum_{n\le \sqrt{x}}{\bf 1}_{h_{m,k}\left(n\right)>0}+\sum_{1\le r\le (\log 4)^{-1}(\log x)}\;\sum_{2^{-r}x< n\le 2^{1-r}x}{\bf 1}_{h_{m,k}\left(n\right)>0}\\
&\ll \sqrt{x}+\sum_{1\le r\le (\log 2)^{-1}(\log x)} \frac{2^{1-r}x}{(\log x)^{\delta}(\log\log x)^{3/2}}\ll \frac{x}{(\log x)^{\delta}(\log\log x)^{3/2}}.
\end{align*}
This completes the proofs.
\end{proof}
%\bibliographystyle{plain}
%\bibliography{test}

\end{document}